\documentclass[a4paper, 12pt]{article}

\usepackage{graphicx} 
\usepackage{amsmath, amsthm, amssymb}
\usepackage{enumitem}

\usepackage[utf8]{inputenc}
\usepackage[a4paper, margin=2.5cm]{geometry}
\usepackage{needspace}
\usepackage{accents}

\usepackage{float}

\usepackage{tikz}
\usetikzlibrary{positioning}

\usepackage{authblk}

\usepackage[disable]{todonotes}
\definecolor{lightblue}{rgb}{0.68,0.85,0.9}
\definecolor{lightgreen}{rgb}{0.59,0.93,0.56}

\newtheorem{theorem}{Theorem}[section]
\newtheorem*{thmnonum}{Theorem}
\newtheorem{definition}{Definition}
\newtheorem{lemma}[theorem]{Lemma}

\title{Domination and packing in graphs}
\author[1]{Ákos Dúcz}
\author[1,2,3]{Anna Gujgiczer}
\affil[1]{Department of Computer Science and Information Theory\\Budapest University of Technology and Economics,
Budapest, Hungary}
\affil[2]{HUN-REN Alfréd Rényi Institute of Mathematics, Budapest, Hungary}
\affil[3]{MTA--HUN-REN RI Lendület "Momentum" Arithmetic Combinatorics Research Group, Budapest, Hungary}
\date{February 2026}

\begin{document}

\maketitle

\begin{abstract}
The dominating number $\gamma(G)$ of a graph $G$ is the minimum size of a vertex set whose closed neighborhoods cover all vertices of $G$, while the packing number $\rho(G)$ is the maximum size of a vertex set whose closed neighborhoods are pairwise disjoint.
In this paper we investigate graph classes $\mathcal{G}$ for which the ratio $\gamma(G)/\rho(G)$ is bounded by a constant $c_{\mathcal{G}}$ for every $G \in \mathcal{G}$. Our main result is an improved upper bound on this ratio for planar graphs. We also extend the list of graph classes admitting a bounded ratio by showing this for chordal bipartite graphs and for homogeneously orderable graphs. In addition, we provide a simple, direct proof for trees.
\end{abstract}

\section{Introduction}
Many graph parameters come in pairs $(p_\ell, p_u)$ where one parameter - say $p_\ell$ - provides a trivial lower bound on the other. A natural question in such situations is whether the gap between the two can grow arbitrarily large, and if so, whether $p_u$ can still be bounded above by some function of $p_\ell$. If this is impossible in general, one may ask whether such a bound exists within specific graph classes.

A classical example is the relationship between the clique number $\omega(G)$ and the chromatic number $\chi(G)$: the search for $\chi$-bounded graph classes $\mathcal{G}$, in which $\chi(G)$ can be upper bounded by a function of $\omega(G)$ for every $G \in \mathcal{G}$, fits exactly into this framework.

In this paper we study another such pair of parameters: the packing number $\rho(G)$ and the domination number $\gamma(G)$.

\begin{definition}[Packing number]
    For a graph $G$, its \emph{packing number} $\rho(G)$ is the maximum size of a vertex set $P \subseteq V(G)$ whose closed neighborhoods are pairwise disjoint.
\end{definition}

\begin{definition}[Domination number]
    For a graph $G$, its \emph{domination number} $\gamma(G)$ is the minimum size of a vertex set $D \subseteq V(G)$ such that every vertex of $G$ lies in the closed neighborhood of some vertex of $D$.
\end{definition}

Observe that the inequality $\rho(G) \leq \gamma(G)$ holds trivially for every graph $G$, since no two vertices in a packing can be dominated by the same vertex in a dominating set. 
Similarly to the situation with the clique number and the chromatic number, the problems of computing $\gamma(G)$ and $\rho(G)$ admit formulations as integer programs, and the linear programming relaxations of these formulations are dual to each other.

\begin{align}\label{eq:LP-relax}
&\text{Domination number } \gamma(G): && \text{Packing number } \rho(G): \nonumber\\
&\text{minimize} \quad \sum_{v\in V(G)} x_v&
&\text{maximize} \quad \sum_{v\in V(G)} y_v \nonumber\\
&\text{s.t.} \quad \sum_{u\in N[v]} x_u \ge 1, \mbox{     }\forall v\in V(G)&
&\text{s.t.} \quad  \sum_{u\in N[v]} y_u\le 1, \mbox{     }\forall v\in V(G), \nonumber\\
&\quad \quad x_v\in \{0,1\}, \quad \forall v\in V(G)&
&\quad \quad y_v\in \{0,1\}, \quad \forall v\in V(G) \nonumber
\end{align}

The fractional versions of the domination and packing numbers are denoted by $\gamma_f(G)$ and $\rho_f(G)$, respectively. Based on the above observations, these four parameters satisfy the following relations:
$$\rho(G)\leq\rho_f(G) = \gamma_f(G)\leq\gamma(G)$$

The gap between the fractional and integral variants of these parameters has also been studied. A well-known result in this direction, due to Lovász~\cite{Lovasz75}, is the following:

\begin{thmnonum}[Lovász \cite{Lovasz75}]
    For every graph $G$, we have $\frac{\gamma(G)}{\gamma_f(G)} \leq log(\Delta(G))$, where $\Delta(G)$ denotes the  maximum degree of $G$.
\end{thmnonum}

It is also a natural question whether better bounds on the ratio for the fractional and integral versions of these parameters can be obtained for special graph classes. In this paper we focus on the gap between $\gamma(G)$ and $\rho(G)$ and in particular we restrict our attention to graph classes for which there exists a constant $c$ such that $\gamma(G) \leq c\,\rho(G)$ holds for every graph in the class. Note, that this bound immediately gives us an integrality bound as well.
Also observe, that a simple general upper bound on the domination--packing ratio can be given by the the maximum degree. Indeed, for any graph $G$, the union of the open neighborhoods of the vertices in a maximum packing forms a dominating set, which implies $\gamma(G)\le \Delta(G)\,\rho(G)$. Our interest is therefore in smaller, degree-independent bounds for specific graph classes.

Throughout the paper we use the notions of neighborhoods and second neighborhoods, so we collect the corresponding notation here for convenience. Unless stated otherwise, all graphs considered in this paper are simple.

\begin{description}
    \item[$dist_G(u,v)$:] The {\it distance} between two vertices $v,u\in V(G)$, denoted by $dist_G(u,v)$ - or if it is clear from the context just $dist(u,v)$ - is the length of the shortest path between $u$ and $v$ in $G$. 
    \item[$N_G(v)$:] The {\it open neighborhood} $N_G(v)$ of a vertex $v \in V(G)$ contains all the vertices $u \in V(G)$ for which $dist_G(u,v)=1$.
    \item[$N_G^2(v)$:] The {\it second open neighborhood} $N^2_G(v)$ of a vertex $v \in V(G)$ contains all the vertices $u \in V(G)$ for which $dist_G(u,v)=1 \text{ or }2$.
    \item[$N_G\text{[}v\text{]}$:] The {\it closed neighborhood} $N_G[v]$ of a vertex $v \in V(G)$ contains all the vertices $u \in V(G)$ for which $dist_G(u,v)\leq1$, so $N_G[v] = N_G(v) \cup \{v\}$.
    \item[$N_G^2\text{[}v\text{]}$:] The {\it second closed neighborhood} $N^2_G[v]$ of a vertex $v \in V(G)$ contains all the vertices $u \in V(G)$ for which $dist_G(u,v)\leq2$, so $N^2_G[v] = N^2_G(v) \cup \{v\}$.
    \item[$N_G\text{[}X\text{]}$:] The {\it closed neighborhood} $N_G[X]$ of a vertex set $X \subseteq V(G)$ is the union of all the closed neighborhoods of $v \in X$, i.e. $N_G[X] = \bigcup_{v\in X}N_G[v]$. 
\end{description}
As in the notation of the distance, we may omit $G$ from the lower index if it is clear from the context.

In the next subsections we first summarize the previous work in this area, and then we present our results.

\subsection{Related work}
Some of the earliest results concern classes where the two parameters coincide. In particular, it is known that $\gamma(G)=\rho(G)$ holds for trees~\cite{MeirMoon75}, for strongly chordal graphs~\cite{Farber84}, and for dually chordal graphs~\cite{Brandstadt98}.

More generally, several graph classes are known to admit a constant bound on the ratio $\gamma(G)/\rho(G)$. Cactus graphs were shown to satisfy $\gamma(G)\le 2\rho(G)$ in~\cite{LiuRautenbachRies13}, and connected biconvex graphs were recently shown to admit the same bounded ratio in~\cite{GG25}.
Constant upper bounds on $\gamma(G)/\rho(G)$ for outerplanar graphs and for bipartite cubic graphs were obtained in~\cite{GG25}, namely $\gamma(G)\le 3\rho(G)$ in the outerplanar case and $\gamma(G)\le \tfrac{120}{49}\rho(G)$ for bipartite cubic graphs.
Further examples of graph classes admitting bounded domination-packing ratios include asteroidal triple-free graphs, convex graphs and unit disk graphs \cite{bonamy2025graph}, whose best known bounds are $3, 3$ and $32$ respectively.

B\"ohme and Mohar~\cite{BohmeMohar03} studied covering and packing by balls (a generalization of domination and packing, corresponding to balls of larger fixed radius) in graphs excluding a complete bipartite graph as a minor; in particular, their results yield constant bounds on $\gamma(G)/\rho(G)$ for several minor-free classes, including graphs of bounded genus and hence planar graphs. 
In \cite{bousquet2021packing} the authors studied in general the case of $H$-minor-free graphs for every graph $H$ where the radius of the balls are not fixed. Their result also implies some large constant bound for planar graphs. In \cite{bonamy2025graph} a constant bound of 10 for planar graphs is given.
A more systematic sufficient condition was later given by Dvo\v{r}\'ak~\cite{Dvorak13} in terms of \emph{weak coloring numbers} ($wcol_1$ and $wcol_2$). Since we will not work directly with this notion, we do not recall its definition here; it suffices to note that bounded weak coloring numbers provide one convenient way to formalize sparsity, and that \cite{Dvorak13} implies that a bound on $wcol_2(G)$ yields a constant bound on $\gamma(G)/\rho(G)$. For further details, we refer the reader to \cite{Dvorak13, hodor2025weak}.

These results suggested that bounded domination--packing ratios might be closely tied to sparsity. However, it was shown in~\cite{bonamy2025graph} that $\gamma(G)/\rho(G)$ is bounded by a constant for $2$-degenerate graphs, which do not fall under these sparsity assumptions.

Very recently, in~\cite{BonamyDvorakMichelMiksanik26} the authors gave an exact characterization of monotone graph classes $\mathcal{G}$ of bounded average degree for which the domination number of every $G\in\mathcal{G}$ is bounded by a linear function of its packing number.

On the negative side, Burger et al.~\cite{Burger09} observed that the Cartesian product of two complete graphs has domination number linear in the number of vertices while its packing number is $1$, implying that the ratio $\gamma(G)/\rho(G)$ can be unbounded, even for bipartite graphs. More recently, Dvořák~\cite{Dvorak19} showed that the ratio is unbounded for the class of graphs with arboricity 3.
Further examples where the domination--packing ratio is unbounded are split graphs, and consequently also chordal graphs \cite{bonamy2025graph}. Moreover, unboundedness holds already for $3$-degenerate graphs \cite{Dvorak13}.

Two natural graph classes remained unresolved for the boundedness of $\gamma(G)/\rho(G)$, namely chordal bipartite graphs and homogeneously orderable graphs (see the summary figure in \cite{bonamy2025graph}).
More broadly, even for many graph classes where boundedness is established, the optimal constant is typically unknown: existing arguments often provide explicit bounds, but their sharpness is open in most cases.

\subsection{Our results}
In this paper we obtain the following results.
\begin{itemize}
    \item We improve the best previously known constant for planar graphs by showing that
    $\gamma(G)\le 7\,\rho(G)$ for every planar graph $G$, strengthening the earlier bound of 10.
    
    \item We settle two previously open cases by proving that $\gamma(G)\le 2\,\rho(G)$ holds for
    every chordal bipartite graph and for every homogeneously orderable graph.
    
    \item We give a simple direct proof for trees. While the statement is classical, to our
    knowledge existing proofs in the literature are stated in the more general setting of
    distance domination and packing and are relatively involved.
\end{itemize}

The organization of the paper is as follows. We first prove the improved bound for planar graphs, which constitutes our main result. Next, we treat the two new graph classes, homogeneously orderable graphs and chordal bipartite graphs. Finally, we give the short proof for trees, and we conclude with a discussion of remaining open
questions.

\newpage
\section{Planar graphs}

The main difficulty of using induction to prove domination/packing bounds is that removed vertices may still affect distances in the resulting smaller graph. A way to circumvent this is by using a stronger induction hypothesis. An example of this is given in \cite{bonamy2025graph}, where the authors define $(X,Y)$-dominating sets and $(X,Y)$-packings for certain subsets $X, Y \subseteq V(G)$. 

Inspired by these methods, we introduce the following definitions:

\begin{definition}[X-domination]
    Let $G$ be a graph and $X \subseteq V(G)$. A set $D\subseteq V(G)$ is an $X$-dominating set of $G$ if $N[D]\cup X = V(G)$. The minimal size of such a $D$ for fixed $X$ is denoted $\gamma_X(G)$.
\end{definition}

We can similarly generalize packing as follows:
\begin{definition}[X-packing]
    Let $G$ be a graph and $X \subseteq V(G)$. A set $P\subseteq V(G)$ is an $X$-packing of $G$ if $X \cap P = \emptyset$ and $N[x] \cap N[y] = \emptyset$ for any two (different) $x, y \in P$. The maximal size of such a $P$ for fixed $X$ is denoted $\rho_X(G)$.
\end{definition}

These definitions can be understood as a simplification of the concepts introduced in \cite{bonamy2025graph}.
In order to prove that $\gamma(G) \leq 7\rho(G)$ for all planar graphs, we will prove the following stronger statement:
\begin{equation}
 \gamma_X(G) \leq 7\rho_X(G)  
\end{equation}

for all planar $G$ and all $X\subseteq V(G)$. Our theorem follows by substituting $X=\emptyset$.

\subsection{Properties of a minimal counterexample}
Let $G$ and $X$ be a minimal counterexample to $(1)$, in the sense that $\vert V(G)\vert + \vert E(G)\vert$ is minimal. We can prove the following lemmas about $(G, X)$:

\begin{lemma}\label{lem:x_independent}
    $X$ is an independent set in $G$.
\end{lemma}
\begin{proof}

Suppose there exist two vertices $x_1, x_2 \in X$ with an edge $e$ between them. Consider $G' = G - e$. By the minimality of $G$, we have $\gamma_{X}(G') \leq 7\rho_{X}(G')$. Therefore we can find an $X$-domination-packing pair $(D',P')$ in $G'$ with $\vert D'\vert \leq 7\vert P'\vert$. 

Since by definition $X \cap P' = \emptyset$ and $x_1, x_2 \in X$, one can easily check that $P'$ is also an $X$-packing in $G$, because the edge $e$ cannot be part of a path of length $1$ or $2$ between two vertices in $P'$. Similarly, since the vertices of $X$ need not be dominated by an $X$-dominating set by definition, $D'$ is also an $X$-dominating set in $G$, a contradiction.
\end{proof}

\begin{lemma}\label{lem:x_deg_7}
    For any vertex $v \in V(G)$ with $d(v) \leq 7$ we have $v \in X$.
\end{lemma}

\begin{proof}

Let $v$ be a vertex s.t. $v \not \in X$ and $d(v) \leq 7$. Let $X' = X \cup N^2(v)$, and $(D',P')$ be an $X'$-domination-packing pair in $G' = G-v$ with $\vert D'\vert \leq 7\vert P' \vert$ (which exists by minimality).

Let $D = D' \cup N(v)$. $D$ is clearly an $X$-dominating set of $G$ with $\vert D\vert \leq \vert D'\vert + 7$, since $N(v)$ dominates all of $N^2[v]$ with at most $7$ vertices, and $D'$ is an $X$-dominating set of the rest of $G$ by definition. Now let $P = P'\cup \{v\}$, and observe that $P$ is an $X$-packing of $G$ since $X'$ contains $N^2(v)$, and because $v\notin X$. Therefore $(D, P)$ is an $X$-domination-packing pair in $G$ with $\vert D\vert \leq 7\vert P \vert$, a contradiction.
\end{proof}

\newpage

\begin{lemma}\label{lem:x_deg_2}
    We have $d(x) \geq 2$ for every vertex $x \in X$.
\end{lemma}

\begin{proof}

Trivially, there are no isolated vertices in $G$. For a vertex $x \in X$ of degree one, consider $G' = G-x$, and ($D'$, $P'$) an $X$-domination-packing pair in $G'$ with $\vert D'\vert \leq 7\vert P' \vert$. Since the addition of a degree one vertex cannot change distances between any other vertices, $P'$ is also an $X$-packing in $G$. But since $x \in X$, $D'$ is also an $X$-dominating set in $G$, again a contradiction.
\end{proof}

\begin{definition}

    Let $x \in X$ be a vertex of $G$.
    We will call a pair $\{a, b\} \subseteq N(x), (a\neq b)$ "bad", if either $(a,b)\in E(G)$ or $\{a,b\} \subseteq N(w)$ for some $w \in X $ with $ w \not = x$, and "good" otherwise.
\end{definition}

\begin{lemma}\label{lem:good_nbr_2}
    If $x \in X$ is a degree two vertex in $G$ with neighbors $a,b \in V(G)$, then $\{a,b\}$ is a good pair.
\end{lemma}

\begin{proof}
Suppose that $x, a, b \in V(G)$ contradict Lemma~\ref{lem:good_nbr_2}.
Consider $G' = G - x$ and $X' = X \setminus x$. Let $(D', P')$ be an $X'$-domination-packing pair in $G'$ with $\vert D'\vert \leq 7\vert P' \vert$ by minimality.

Clearly $D'$ is still an $X$-dominating set in $G$ because $x \in X$. Since adding $x$ to $G'$ does not change the distances between any two vertices in $V(G')$ (because $(a,b)$ is a bad pair), we have that $P'$ is an $X$-packing in $G$, which is a contradiction.
\end{proof}

\begin{lemma}\label{lem:good_nbr_3}
    If $x \in X$ is a degree three vertex in $G$, then there are at least two good pairs in $N(x)$.
\end{lemma}

\begin{proof}

Suppose this fails. Since $d(x) = 3$, $x$ has at least $2$ bad neighbor pairs, and these pairs have some common vertex. We will denote the two bad pairs $\{n_1, n_2\} \subset N(x)$ and $\{n_1, n_3\} \subset N(x)$ with the common vertex being $n_1$ WLOG. Consider $G' = G-(x,n_1)$, and $(D',P')$ an $X$-domination-packing pair in $G'$ with $\vert D'\vert \leq 7\vert P' \vert$ by minimality. $D'$ is an $X$-dominating set of $G$ because adding back the edge $(x,n_1)$ cannot ruin $X$-domination. But since $\{n_1, n_2\}$ and $\{n_1, n_3\}$ are bad pairs, at most one of $n_1, n_2, n_3$ can be in $P'$. We also have $x \not \in P'$ because $x \in X$, and therefore the addition of $(x, n_1)$ to $G'$ does not ruin the $X$-packing property of $P'$. Since $(D',P')$ is an $X$-domination-packing pair in $G$ with $\vert D'\vert \leq 7\vert P' \vert$, we have a contradiction.
\end{proof}

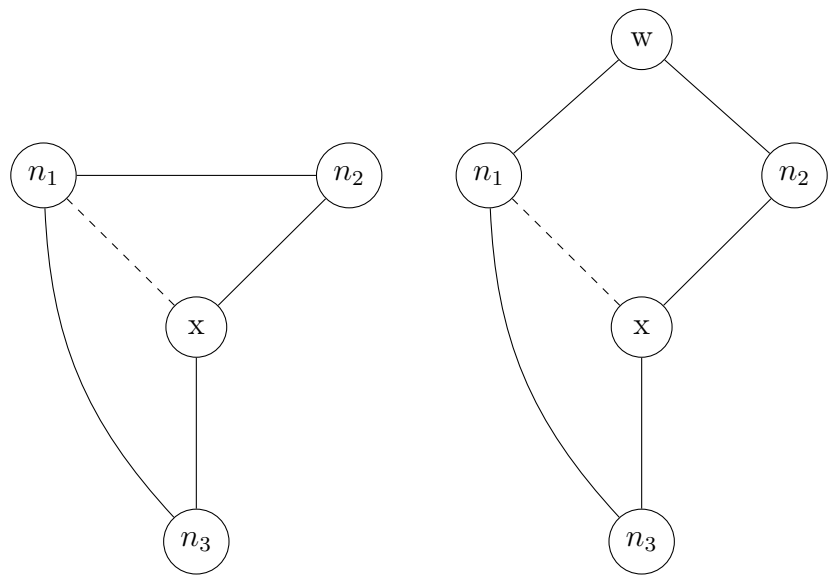
\begin{figure}[H]
\centering
\begin{tikzpicture}[
  vertex/.style={circle, draw, minimum size=8mm},
  edge/.style={-}
]

\node[vertex] (v0) {x};
\node[vertex] (v1) [above left=2cm of v0] {$n_1$};
\node[vertex] (v2) [above right=2cm of v0] {$n_2$};
\node[vertex] (v3) [below=2cm of v0] {$n_3$};

\draw[edge, dashed] (v0) -- (v1);
\draw[edge] (v0) -- (v2);
\draw[edge] (v0) -- (v3);

\draw[edge] (v1) -- (v2);

\draw[edge] (v1) to[bend right=20] (v3);

\end{tikzpicture}
\qquad
\begin{tikzpicture}[
  vertex/.style={circle, draw, minimum size=8mm},
  edge/.style={-}
]

\node[vertex] (v0) {x};
\node[vertex] (v1) [above left=2cm of v0] {$n_1$};
\node[vertex] (v2) [above right=2cm of v0] {$n_2$};
\node[vertex] (v3) [below=2cm of v0] {$n_3$};
\node[vertex] (v4) [above=3cm of v0] {w};

\draw[edge, dashed] (v0) -- (v1);
\draw[edge] (v0) -- (v2);
\draw[edge] (v0) -- (v3);

\draw[edge] (v1) -- (v4);
\draw[edge] (v2) -- (v4);

\draw[edge] (v1) to[bend right=20] (v3);

\end{tikzpicture}

\caption{Possible configurations of $N(v)$ in the proof of Lemma~\ref{lem:good_nbr_3} (not exhaustive).}
\end{figure}

\subsection{Discharging method}

\begin{lemma}\label{lem:triangulate}
    If $H$ is a simple connected planar graph and $I$ is and independent set in $H$, then we can add new edges to $H$ such that:
    \begin{itemize}
        \item $I$ is still independent in the new graph $H'$.
        \item $H'$ is planar (but not necessarily simple).
        \item Every face of $H'$ is a triangle.
    \end{itemize}
\end{lemma}
\begin{proof}

    Let $F$ be a non-triangular face of $H$. It is easy to see that $F$ can be subdivided by adding a new edge in such a way that both new faces have at least 3 bounding edges, and such that $I$ remains independent. Repeating this process until all faces are triangles gives our lemma.
\end{proof}
\emph{Remark.} To illustrate that $H'$ is not necessarily simple, consider triangulating $C_4$ with $I$ being two opposite vertices: the triangulation of the inner and outer face will yield the addition of edges between the same two vertices, $V(C_4)\setminus I$.

\begin{lemma}\label{lem:discharge}
    Let $H$ be a simple planar graph with minimum degree $4$. Then $H$ has an edge $(x, y)$ with $d(x) \leq 7$ and $d(y) \leq 7$.
\end{lemma}

\begin{proof}
    We may assume that $H$ is connected (otherwise we just take some component of $H$).
    Let $I$ denote the set of vertices in $H$ with $d(v) \leq 7$. If $I$ is not independent, the statement follows. Suppose that $I$ is independent. By Lemma \ref{lem:triangulate}, we can add new edges to $H$ to obtain $H'$, every face of which is a triangle and in which $I$ is still independent.

    Let us assign a charge of $d(v) - 6$ to every vertex $v$ of $H'$. By Euler's formula, we have that the total sum of these charges is negative. 
    Now let each vertex $v$ with $d(v) \geq 8$ give a charge of $1/2$ to each of it's neighbors in $I$ (if there are multiple edges between $v$ and $w$, then we count $w$ as a 'neighbor' multiple times). Since $H'$ is triangulated and $I$ is independent in $H'$, such a vertex $v$ gives charge to at most half of it's neighbors. Thus every vertex with $d(v) \geq 8$ has nonnegative charge after redistribution.

    If $d(v) \leq 7$ then all of its neighbors have degree at least $8$, therefore $v$ must receive $1/2$ charge from each neighbor. The initial charge of $v$ is smallest in case $d(v) = 4$, which gives a charge of $d(v)-6 + 4\cdot(1/2) = 0$ after discharging. Therefore every vertex has nonnegative charge at the end of the procedure, which is a contradiction. 
\end{proof}

\subsection{Finishing the proof}

Let $X^{\leq3} = \{x\, \vert \, x\in X, d(x) \leq 3\}$. By Lemma \ref{lem:x_deg_2}, this set only contains vertices of degrees $2$ and $3$. By Lemma \ref{lem:x_deg_7}, we also have that every $v\in V(G)$ with $d(v) \leq 3$ is in $X^{\leq3}$.

After fixing a planar embedding of $G$, we apply the following operation:
for each vertex $x \in X^{\leq3}$, we select a maximal set of good pairs $P_x$ in $N(x)$. Note that $\vert P_x\vert \geq d(x)-1$ by Lemmas \ref{lem:good_nbr_2} and \ref{lem:good_nbr_3} above.
We now remove every vertex $x \in X^{\leq3}$ from $G$ and simultaneously add every edge in $P_x$ to $E(G)$. We denote the resulting graph by $H$, and prove the following statements:

\begin{enumerate}[label=(\arabic*)]
    \item The remaining vertices of $X$ are independent in $H$.
    \item $d_H(v) \geq d_G(v)$ for all $v \in V(H)$.
    \item $H$ is simple and planar.
\end{enumerate}

To prove $(1)$, we need to check that no new edges were added between two vertices in $X$. Indeed, since the new edges in $H$ are all good pairs in some $N(x)$ with $x \in X$, and because $X$ is independent in $G$, we cannot have added an edge between two vertices in $X$.

The second statement follows from Lemmas \ref{lem:good_nbr_2} and \ref{lem:good_nbr_3}. When removing a vertex $x \in X^{\leq3}$ during our operation, we decrease the degrees of its neighbors by one. However, we also add at least $d(x)-1$ new edges to $N(x)$, and it is easy to check that this increases the degrees of each neighbor by at least one, since $d(x) \leq 3$.

In $(3)$, simplicity follows from the definition of a good pair, and planarity is easy to check since $d(x) \leq 3$ for each removed vertex.

From $(1)$ and $(2)$ it follows that the set of vertices in $H$ with $d(v) \leq 7$ is independent. But $H$ is also simple and planar by $(3)$, which contradicts Lemma \ref{lem:discharge}, as $H$ has minimum degree $4$. Thus our minimal counterexample $(G, X)$ cannot exist, and the theorem is proved. \qed

\section{Homogeneously orderable graphs}

In order to talk about homogeneously orderable graphs, we first need to define the following notions.
Let $G$ be an undirected graph. We call a set $A \subseteq V(G), A \not= \emptyset $ homogeneous, iff every vertex of $A$ has the same neighborhood in $V(G)\setminus A$. 

A vertex $v$ of $G$ is h-extremal, if there exists a set $D \subseteq N^2[v]$ such that $D$ is homogeneous in $G$, and $D$ dominates $N^2[v]$. In this case, $D$ is a homogeneous dominating set of $N^2[v]$.

We need the following lemma from \cite{brandstadt1997homogeneously} about homogeneous dominating sets. We also include the proof here for completeness.

\begin{lemma}{\cite{brandstadt1997homogeneously}}\label{lem:hom_ord_n}
Let $v$ be a h-extremal vertex of $G$. Then there exists a homogeneous dominating set $D'$ of $N^2[v]$, such that $D' \subseteq N[v]$. 
\end{lemma}

\begin{proof}
    Let $D$ be any homogeneous dominating set of $N^2[v]$. If $v \notin D$, then some vertex $x \in N(v) \cap D$ must dominate $v$, but then $D \subseteq N(v)$ since $D$ is homogeneous, so $D' = D$ suffices.

    If $v \in D$, then let $H = (N^2[v] \setminus N[v])$. If $H = \emptyset$, we are done since $D' = D$ suffices. Because $D$ is homogeneous, no $x \in D$ can have a neighbor in $H\setminus D$ (since $v$ has no neighbor in $H$). But $H$ must be dominated by $D$, therefore $H \subseteq D$. By homogeneity, this implies that every vertex in $H$ is connected to every vertex of $N(v)\setminus D$, since $v \in D$. We also have that every vertex of $N(v)\cap D$ is connected to every vertex of $N(v) \setminus D$ by homogeneity. But then $D' = N(v)\setminus D$ is a homogeneous dominating set of $N^2[v]$.
\end{proof}

A graph $G$ is homogeneously orderable, if there exists an ordering of its vertices such that every vertex $v$ is h-extremal in the subgraph induced by $v$ and all subsequent vertices in this order. We call this vertex order the homogeneous ordering of $G$.

\begin{theorem}
    $\gamma(G) \leq 2\cdot\rho(G)$ for any homogeneously orderable graph $G$.
\end{theorem}

\begin{proof}
    Choose a homogeneous ordering of $G$, and index $V(G) = \{v_1, v_2,...,v_n\}$ using this ordering. Let $i(v)$ denote the index of a vertex $v$, so $i(v_k)=k$.

    Let $P$ be a maximal packing in $G$ with $ \sum_{v\in P}i(v) $ as small as possible. For each $v_k\in V(G)$, we have that $v_k$ is h-extremal in the subgraph induced by $V(G)^{\geq k}=\{v \, | \, v \in V(G), i(v) \geq k\}$. Then by Lemma~\ref{lem:hom_ord_n}, $N[v]$ contains a homogeneous dominating set of $N^2[v]$ in this graph. Let $f(v_k)$ be any arbitrary vertex in this dominating set. Note, that $\{v_k, f(v_k)\}$ dominates $N^2[v_k]$ in $G^{\geq k}$ by the definition of homogeneous dominating set.
    
    Now set $D=\{f(v)\, | \, v\in P\}\cup P$, and suppose there is a vertex $w$ in $G$ that is not dominated by $D$. Clearly $w \not\in N[P]$. Then there must be some $z \in P$ with $dist(w, z) = 2$, otherwise $P$ would not be maximal.
    
    We cannot have $i(z) < i(w)$, since then $D$ would dominate $w$ by definition. Therefore $i(w) < i(z)$. Observe that there cannot exists any other $z' \not= z$ with $dist(w,z') = 2, i(w) < i(z')$ and $z' \in P$, since then $f(w)$ would be adjacent to both of them, contradicting $z, z' \in P$. Since $z$ is unique with this property, $P' = (P\setminus\{z\})\cup\{w\}$ is also a packing in $G$, which contradicts the minimality of $P$ in the index sum.
\end{proof}

The constant $2$ is best possible, as evidenced by $C_4$.

\section{Chordal bipartite graphs}

A bipartite graph $G=(A,B,E)$ is called chordal bipartite, if every cycle of $G$ with length at least $6$ has a chord. A graph $G$ is strongly chordal if it is chordal, and every cycle $C$ of $G$ with even length at least $6$ has an odd chord (a chord connecting two vertices of odd distance apart in $C$).

\begin{lemma}{\cite{brandstadt1999graph}}\label{lem:split_lemma}
    Let $G=(A,B,E)$ be a bipartite graph, and let $split_A(G)$ be the graph obtained from $G$ by making $A$ a clique. Then $split_A(G)$ is strongly chordal iff $G$ is chordal bipartite.
\end{lemma}

    For proving a bounded domination-packing ratio for chordal bipartite graphs we will only need the direction of this above mentioned lemma, that if $G$ is chordal bipartite then $split_A(G)$ is strongly chordal. We include the proof of this for completeness.

\begin{proof}
    If $G$ is chordal bipartite, then $G' = split_A(G)$ is clearly chordal, since any cycle of $G'$ with length at least $4$ must have at least two nonconsecutive vertices in $A$. Let $C$ be a cycle of $G'$ with an even length of at least $6$. If $C$ contains no newly added edges (those in $E(G')\setminus E(G)$), then it must have an odd chord since $G$ is chordal bipartite and $E(C) \subseteq E(G)$.
    
    If $C$ does contain some new edges, then by parity it must contain an even number of them. If there are at least four of these, then we can choose two such that they have no common vertex, and if there are exactly two then we will choose those. WLOG let the chosen edges be $(x,y), (z,w) \in E(C)$ with $x,y,z,w \in A$, and such that $y = z$ is the common vertex if there is one. If $y\not=z$, then one of $(x,z), (x,w), (y,z), (y,w)$ must be an odd chord, since $C$ has length at least $6$. Otherwise $(x,y)$ and $(y, w)$ are the only new edges in $C$, therefore $(y, v)$ is an odd chord of $C$ for any other $v \in V(C) \cap A \setminus\{ x,w\}$.
\end{proof}

Using this lemma, and the fact that $\gamma(G)/\rho(G)$ equals $1$ for strongly chordal graphs \cite{Farber84}, we prove the following theorem:

\newpage

\begin{theorem}
    $\gamma(G) \leq 2\cdot\rho(G)$ for every chordal bipartite graph $G$.
\end{theorem}

\begin{proof}
    Let $G = (A, B, E)$ and let $G_A=split_A(G)$ and $G_B = split_B(G)$ be defined as in Lemma \ref{lem:split_lemma}. Since $G_A$ and $G_B$ are strongly chordal, we have $\gamma(G_A) = \rho(G_A)$ and similarly for $G_B$. 
    
    Let $D_A, D_B$ be minimal dominating sets and $P_A, P_B$ be maximum packings in $G_A$ and $G_B$, respectively. Since $P_A$ and $P_B$ are also packings in $G$, we have $\rho(G) \geq max(|P_A|, |P_B|)$. Now observe that $D_A \cup D_B$ is a dominating set in $G$: let $v$ be an arbitrary vertex of $G$, WLOG say $v \in A$. Then the vertex $d_B \in D_B$ dominating $v$ in $G_B$ will dominate it in $G$ as well. Therefore $\gamma(G) \leq 2\cdot max(|D_A|, |D_B|)$. Since $|D_A| = |P_A|$ and $|D_B| = |P_B|$, our theorem follows.
\end{proof}

The constant $2$ is best possible, as evidenced again by $C_4$.

\section{Trees}

To our knowledge, the only proofs concerning trees in literature provide bounds for general distance domination and packing, and are relatively involved.
Here, we give a simple, direct proof of the following theorem:

\begin{theorem}
    $\gamma(T) = \rho(T)$ for every tree $T$.
\end{theorem}

\begin{proof}
    We only need to prove $\gamma(T) \leq \rho(T)$. Choose some vertex $v \in V(T)$ as the "root", and suppose that $P$ is a maximal packing of $T$ with $ \sum_{p\in P} dist(v, p) $ as large as possible. For each $w \in V(T)$, we can define $parent(w)$ as the first vertex on the path from $w$ to the root. In case $w = v$, we define $parent(v) = v$.

    Consider $D = \{ parent(p) | p \in P \}$. Suppose that there is some vertex $w$ not dominated by $D$. Since $w \notin P$, there is some $p_w \in P$ at distance at most $2$ from $w$ (otherwise $P$ would not be maximal). Notice that $p_w$ must either be $parent(w)$ or $parent(parent(w))$, since otherwise $w$ would be dominated by $D$. Since both of these can't be in $P$, we have that $p_w$ is the only vertex in $P$ with $dist(p_w, w) \leq 2$. But since $dist(v,p_w) < dist(v, w)$, $(P\setminus\{p_w\}) \cup \{w\}$ (which is also a packing) contradicts the maximality of $P$ in the distance sum.
\end{proof}

\section{Summary}
In this work we studied graph classes $\mathcal G$ for which the domination--packing ratio $\gamma(G)/\rho(G)$ is bounded by a constant for all $G\in\mathcal G$.
We settle two previously open cases by proving a constant bound of $2$ for chordal bipartite graphs and for homogeneously orderable graphs.
We also improve the best explicit constant known for planar graphs: we show that every planar graph satisfies $\gamma(G)\le 7\,\rho(G)$, improving the previously best bound $10$. This constitutes our main technical contribution.

For chordal bipartite graphs and for homogeneously orderable graphs, our bound of $2$ is sharp.
In contrast, we do not expect the planar constant to be optimal. A natural remaining question is therefore whether the planar bound can be improved further.
The best possible upper bound for planar graphs is $3$, which can be seen by taking the graph in Figure 2. This graph was given in \cite{goddard2002domination} as an example for $G$ having domination number 3 and radius 2 (therefore packing number 1).

\begin{figure}[H]
    \centering
    \includegraphics[width=0.3\linewidth]{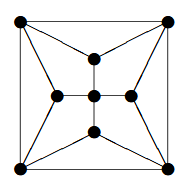}
    \caption{A planar graph with domination number $3$ and packing number $1$.}
    \label{fig:placeholder}
\end{figure}

We conjecture that the constant bound for planar graphs is $3$, i.e. that every planar graph $G$ satisfies $\gamma(G)\le 3\,\rho(G)$.

The problem of bounding $\gamma(G)/\rho(G)$ remains open for many geometric graph classes as well, such as intersection graphs of axis-parallel rectangles in the plane and of axis-parallel boxes in three dimensions.

\section{Acknowledgments}
Anna Gujgiczer was supported by the Lend\"ulet "Momentum" program of the Hungarian Academy
of Sciences (MTA).

\bibliographystyle{alpha}
\bibliography{dompack}

\end{document}